\documentclass[11pt]{amsart}
\usepackage{amscd}
\usepackage{verbatim}
\usepackage{amssymb}
\usepackage{amsmath}
\usepackage{amsthm}
\usepackage{amsfonts}
\usepackage{mathrsfs}
\usepackage{amsthm}
\usepackage{euscript}
\usepackage{xcolor}

\usepackage[margin=30 mm,heightrounded=true,centering]{geometry}
\newtheorem{theorem}{Theorem}[section]

\newtheorem{lemma}[theorem]{Lemma}

\newtheorem{corollary}[theorem]{Corollary}
\newtheorem{definition}[theorem]{Definition}
\newtheorem{example}[theorem]{Example}

\DeclareMathOperator{\riem}{Riem}
\DeclareMathOperator{\ric}{Ric}
\DeclareMathOperator{\hess}{Hess}
\DeclareMathOperator{\tcd}{TCD}
\DeclareMathOperator{\id}{id}
\DeclareMathOperator{\tr}{tr}
\DeclareMathOperator{\arctanh}{arctanh}

\newcounter{mnotecount}[section]

\begin{document}

\title[$N$-Bakry-\'Emery spacetimes]{Cosmological singularity theorems and splitting theorems for $N$-Bakry-\'Emery spacetimes}

\author{Eric Woolgar}
\address{Department of Mathematical and Statistical Sciences,
University of Alberta, Edmonton, Alberta, Canada T6G 2G1}
\email{ewoolgar(at)ualberta.ca}

\author{William Wylie}
\address{215 Carnegie Building, Department of Mathematics, Syracuse University, Syracuse, NY 13244, USA}
\email{wwylie@syr.edu}

\date{\today}

\begin{abstract}
\noindent We study Lorentzian manifolds with a weight function such that the $N$-Bakry-\'Emery tensor is bounded below. Such spacetimes arise in the physics of scalar-tensor gravitation theories, including Brans-Dicke theory, theories with Kaluza-Klein dimensional reduction, and low-energy approximations to string theory. In the ``pure Bakry-\'Emery'' $N= \infty$ case with $f$ uniformly bounded above and initial data suitably bounded, cosmological-type singularity theorems are known, as are splitting theorems which determine the geometry of timelike geodesically complete spacetimes for which the bound on the initial data is borderline violated. We extend these results in a number of ways. We are able to extend the singularity theorems to finite $N$-values $N\in (n,\infty)$ and $N\in (-\infty,1]$. In the $N\in (n,\infty)$ case, no bound on $f$ is required, while for $N\in (-\infty,1]$ and $N= \infty$, we are able to replace the boundedness of $f$ by a weaker condition on the integral of $f$ along future-inextendible timelike geodesics. The splitting theorems extend similarly, but when $N=1$ the splitting is only that of a warped product for all cases considered. A similar limited loss of rigidity has been observed in prior work on the $N$-Bakry-\'Emery curvature in Riemannian signature when $N=1$, and appears to be a general feature.

\end{abstract}

\maketitle

\section{Introduction}
\setcounter{equation}{0}

\noindent Riemannian and Lorentzian $n$-manifolds with a preferred twice-differentiable function $f:M\to {\mathbb R}$ (sometimes defined in terms of a density) admit a family of generalizations of the Ricci tensor $\ric$, known as the $N$-Bakry-\'Emery-Ricci tensor, or simply the $N$-Bakry-\'Emery tensor, given by
\begin{equation}
\label{eq1.1}
\ric^N_f:=\ric+\hess f-\frac{df\otimes df}{N-n}\ .
\end{equation}
Here $\hess$ denotes the Hessian defined by the Levi-Civita connection $\nabla$ of the metric $g$ by $\hess u :=\nabla^2 u$. The so-called \emph{synthetic dimension} $N\in {\mathbb R}$, $N\neq n$, is the family parameter for the family of tensors (some authors use $m:=N-n$ as the parameter). There is also a tensor called simply the Bakry-\'Emery-Ricci (or more simply Bakry-\'Emery) tensor, given by
\begin{equation}
\label{eq1.2}
\ric_f:=\ric+\hess f\ .
\end{equation}
This tensor is sometimes thought of formally as the $N=\infty$ case of the $N$-Bakry-\'Emery-Ricci tensor. It can equally well be thought of as the $N= -\infty$ case; the sign has no significance.

For Lorentzian manifolds, which will be the focus of this paper, we give the following definition which is analogous to what is done in the Riemannian context.
\begin{definition}\label{definition1.1}
If $\ric^N_f(X,X)\ge \lambda$ for all unit timelike vectors $X$ (i.e., $g(X,X)=-1$) and given functions $f$ and $\lambda$, we say that the \emph{$\tcd(\lambda,N)$ condition} holds for $(M,g,f)$. We call this a \emph{timelike curvature-dimension condition}. We use $\tcd(\lambda)$ to denote the condition $\ric_f(X,X) \ge \lambda$.
\end{definition}

The $\tcd(\lambda,N)$ condition reduces to the \emph{timelike convergence condition} in general relativity when the function $f$ is constant. In general relativity, the timelike convergence condition is equivalent to the \emph{strong energy condition} for matter when the Einstein equations hold with matter source terms (but without a cosmological term, meaning here that $\lambda=0$). Energy conditions are required in the proofs of the canonical theorems of mathematical relativity, including the singularity theorems (see, e.g., \cite{HE}). One therefore expects to be able to prove similar singularity theorems assuming instead some form of $\tcd(\lambda,N)$ condition. This was done for the $\tcd(0,N)$ condition in \cite{Case} when $N>n$ and when $N=\infty$. Results under the $\tcd(\lambda)$ assumption for $\lambda\ge 0$ and $N=\infty$ appear in \cite{RW, GW}.

In this paper, we extend the results of \cite{GW} to finite values so as to cover all $N\in (-\infty,1]\cup(n,\infty)\cup\{\infty\}$. %As well, we also extend the timelike splitting theorem in \cite{Case} to all $N\in (-\infty,1]\cup(n,\infty)\cup\{\infty\}$.
When $N=\infty$, the theorems of \cite{Case} and \cite{GW} require that $f$ be bounded above. We weaken this to an integral condition on $f$, and require this condition when considering $N\in (-\infty,1]$ as well. However, it is not needed for our extension to $N>n$ of the results in \cite{GW}, just as it was not needed in \cite{Case} for finite values $N>n$. The integral condition on $f$ is expressed in terms of the following definition.

\begin{definition}\label{definition1.2}
We say a future-inextendible timelike geodesic $\gamma:[0,T)\to M$, $T\in (0,\infty]$, is \emph{future $f$-complete} if it is complete with respect to the parameter $s(t):=\int\limits_0^t e^{-\frac{2f(\tau)}{(n-1)}}d\tau$, where we abbreviate $f(\tau):=f\circ\gamma(\tau)$.
\end{definition}

Future $f$-completeness along $\gamma$ is equivalent to the surjectivity of $s:[0,T)\to [0,\infty)$. This definition is motivated by \cite[Definition 6.2]{Wylie} in the Riemannian case. We take note of the following very simple result, which implies that future $f$-completeness is a generalization of the commonly-used condition that $f$ be uniformly bounded above.

\begin{lemma}\label{lemma1.3}
Let $\Sigma$ be a Cauchy surface. Say that $f$ is uniformly bounded above in the future of $\Sigma$. Then each future-complete timelike geodesic is future $f$-complete. In particular, if $\Sigma$ is compact and $\nabla f$ is future-causal in the future of $\Sigma$, then each future-complete timelike geodesic is future $f$-complete.
\end{lemma}

(By \emph{future-causal}, we mean that $g(\nabla f, \nabla f)\le 0$ and $\nabla f$ is future-pointing wherever it is non-vanishing, but we permit it to vanish.)

\begin{proof}
Every future-inextendible timelike geodesic must meet $\Sigma$ exactly once. Beyond the point at which the geodesic intersects $\Sigma$, we have that $f\le k$, so $s(t)\ge e^{-\frac{2k}{(n-1)}}t \to\infty$ as $t\to\infty$, proving the first statement. To prove the second, since $\nabla f$ is future causal beyond $\Sigma$, $f$ must be (weakly) decreasing along the geodesic beyond $\Sigma$, so $f\le \max_{\Sigma} f=:k$ and hence $f$ is uniformly bounded to the future of $\Sigma$.
\end{proof}

With this in hand, we can now state our results. Our first theorem allows us to extend to finite $N$-values and to spacetimes whose future-complete timelike geodesics are future $f$-complete a singularity theorem that was proved in \cite{GW} for $N=\infty$ and $f$ bounded above.

\begin{theorem}\label{theorem1.4}
Let $M$ be a spacetime satisfying $\tcd(0,N)$ for some fixed $N\in (-\infty,1]\cup(n,\infty)\cup \{ \infty \}$. Let $S$ be a smooth compact spacelike Cauchy surface for $M$ with strictly negative $f$-mean curvature
\begin{equation}
\label{eq1.3}
H_f(S) :=H-\nabla_{\nu}f < 0\ ,
\end{equation}
where $\nu$ is the future unit normal field to $S$ and $H$ is the mean curvature with respect to $\nu$. If $N\in (-\infty,1]\cup \{ \infty \}$ suppose further that each future-complete timelike geodesic orthogonal to $S$ is future $f$-complete. Then every timelike geodesic is future incomplete.
\end{theorem}

This theorem is of cosmological type in that \emph{every} timelike geodesic is future incomplete, suggesting a ``Big Crunch''. In cosmology, such theorems are often phrased in time-reversed form so as to suggest the existence of a so-called ``Big Bang''.

The necessity of a condition controlling $f$ when $N\in (-\infty,1]\cup \{ \infty \}$ is clear from the following.
\begin{example}\label{example1.5}
The Einstein static universe $-dt^2+g(S^{n-1},{\rm can})$ in $n>2$ dimensions with $f=e^t$ has $\ric_f(X,X)\ge e^t+\frac{e^{2t}}{(n-N)}>0$ so $\tcd(0,N)$ holds, while $H_f= -e^t < 0$ for any constant-$t$ hypersurface. But this spacetime is geodesically complete.
\end{example}
This example does not violate Theorem \ref{theorem1.4} because $\int\limits_0^{\infty} e^{-\frac{2f(t)}{(n-1)}}dt<\infty$ when $f(t)=e^t$. The spacetime admits future-complete timelike geodesics that are not future $f$-complete and $f$ is not bounded above.

Next we similarly extend a theorem of \cite{GW} applicable to spacetimes with positive cosmological constant.

\begin{theorem}\label{theorem1.6}
Let $M$ be a spacetime having smooth compact Cauchy surface $S$. Suppose that
\begin{enumerate}
\item[(a)] $N>n$ and
\begin{enumerate}
\item[(i)] $\tcd\left ( -(n-1),N\right )$ holds and
\item[(ii)] $H_f< -(n-1)$ on $S$, or
\end{enumerate}
\item[(b)] $N\in (-\infty,1]\cup \{ \infty \}$ and
\begin{enumerate}
\item[(i)] $\tcd\left ( -(n-1)e^{-\frac{4f}{(n-1)}},N\right )$ holds on $M$ to the future of $S$,
\item[(ii)] each future-complete timelike geodesic orthogonal to $S$ is future $f$-complete, and
\item[(iii)] $H_f< -(n-1)e^{-\frac{2B}{(n-1)}}$ on $S$, with $B:=\inf_S f$.
\end{enumerate}
\end{enumerate}
Then every timelike geodesic is future incomplete.
\end{theorem}

We have already noted that the future $f$-complete condition controlling $f$ when $N\le 1$ or $N=\infty$ is implied if there is an upper bound on $f$. Combining this with Theorem \ref{theorem1.6}, it is easy to obtain the following singularity theorem.

\begin{theorem}\label{theorem1.7}
Let $M$ be a spacetime having a smooth compact Cauchy surface $S$. Suppose that $N\in (-\infty,1]\cup \{ \infty \}$. If
\begin{enumerate}
\item[(i$'$)] $\tcd\left ( -(n-1),N\right )$ holds to the future of $S$,
\item[(ii$'$)] $f\le k$ to the future of $S$ for some $k\in {\mathbb R}$, and
\item[(iii$'$)] $H_f< -(n-1)e^{\frac{2(k-B)}{(n-1)}}$ on $S$, with $B:=\inf_S f$,
\end{enumerate}
then every timelike geodesic is future incomplete. Furthermore, conditions (\rm{ii}$'$) and (\rm{iii}$'$) can be replaced by
\begin{enumerate}
\item[(ii$''$)] $\nabla f$ is future causal to the future of $S$ and
\item[(iii$''$)] $H_f< -(n-1)$ on $S$.
\end{enumerate}
\end{theorem}

The above theorems are derived from focusing lemmata that modify similar lemmata used in \cite{GW}. Once the focusing results are derived, the theorems themselves follow along standard lines.

An important interest for us is the borderline case where the inequality assumption on mean curvature is replaced by equality. In this case, splitting theorems were found in \cite{GW}. We can also obtain such theorems for $N< 1$ with $\tcd(\lambda)$ is replaced by $\tcd(\lambda, N)$, but in the $N=1$ case with $\lambda=0$ there is a notable difference as the following example shows. The calculations are exactly the same as those in  the Riemannian case \cite{Wylie}.

\begin{example}\label{example1.8} Let $f: \mathbb{R} \rightarrow \mathbb{R}$ be a uniformly bounded function with uniformly bounded first and second derivatives. Then for any $n \geq 2$ there is a $\lambda$ large enough such that  the metric $-dt^2 \oplus e^{2f(t)/(n-1)}g_{S_{\lambda}}$  satisfies $\tcd(0,1)$ where $g_{S_{\lambda}}$ is the standard metric on the sphere of constant curvature $\lambda$.  Moreover the surfaces $\{c\} \times S$ satisfy $H_f(S) = 0$.
\end{example}

For $\lambda=0$, we have the following theorem.

\begin{theorem}\label{theorem1.9}
Let $S$ be a smooth compact Cauchy surface for $(M,g)$ having $f$-mean curvature $H_f(S) \le 0$. Let $(J^+(S),g)$ satisfy $\tcd(0,N)$ for a fixed $N\in (-\infty,1]\cup (n,\infty)\cup \{ \infty \}$. Assume that each timelike geodesic orthogonal to $S$ is future complete and, if $N\in (-\infty,1]\cup \{ \infty \}$, also future $f$-complete. Then $(J^+(S),g)$ splits as follows, where $h$ is the induced metric on $S$.
\begin{itemize}
\item[(i)] If $N\in (-\infty,1)\cup (n,\infty)\cup \{ \infty \}$, then $(J^+(S),g)$ is isometric to $([0,\infty) \times S, -dt^2 \oplus h)$ and $f$ is independent of $t$.
\item[(ii)] If $N=1$ then $(J^+(S),g)$ is isometric to $([0,\infty) \times S, -dt^2 \oplus e^{2\psi(t)/(n-1)}h)$, and $f$ separates as $f(t,y)=\psi(t)+\phi(y)$, $y\in S$.
\end{itemize}
\end{theorem}

For $\lambda=-(n-1)$, we obtain an analogous result to the warped product splitting found in \cite{GW}. As this splitting is already that of a warped product even for $N\neq 1$, there is no further weakening of rigidity. However, we will have to assume control of $f$ (namely that $\nabla f$ is future causal), even when $N>n$.

\begin{theorem}\label{theorem1.10}
Let $S$ be a smooth compact Cauchy surface for $(M,g)$ having $f$-mean curvature $H_f(S) \le -(n-1)$. Assume that each future-timelike geodesic orthogonal to $S$ is future complete and that $(J^+(S),g)$ satisfies $\tcd(-(n-1),N)$ for some fixed $N\in (-\infty,1]\cup (n,\infty) \cup \{ \infty \}$. Suppose  that $\nabla f$ is future causal. Then $(J^+(S),g)$ is isometric to the warped product $([0,\infty) \times S, -dt^2 \oplus  e^{-2t} h)$, where $h$ is the induced metric on $S$, and $f$ is constant.
\end{theorem}

When $N>n$ we also have the following warped product splitting that does not require that  $\nabla f$ be future casual.

\begin{theorem}\label{theorem1.11}
Let $S$ be a smooth compact Cauchy surface for $(M,g)$ having $f$-mean curvature $H_f(S) \le -(N-1)$  for some $N>n$. Assume that each timelike geodesic orthogonal to $S$ is future complete and that $(J^+(S),g)$ satisfies $\tcd(-(N-1),N)$ for some fixed $N\in (-\infty,1]\cup (n,\infty) \cup \{ \infty \}$. Then $(J^+(S),g)$ is isometric to the warped product $([0,\infty) \times S, -dt^2 \oplus  e^{-2t} h)$, where $h$ is the induced metric on $S$, and $f = (N-n)t + f_S$ where $f_S$ is a function that does not depend on $t$.
\end{theorem}

We have no restriction on the spacetime dimension $n\ge 2$.

In terms of the Brans-Dicke theory of scalar-tensor gravitation in $n=4$ spacetime dimensions \cite{BD, Faraoni, Woolgar}, we may characterize our results as follows. The Brans-Dicke parameter can take values $\omega\in (-3/2, \infty)$. The values $\omega\in [-1,\infty)$ were discussed in \cite{GW}. Our $N\le 1$ results cover the region $\omega\in [-4/3, -1)$. Interestingly, the critical case of $N=1$ corresponds to $\omega=-4/3$. In contrast, $\omega\searrow -\frac32$ corresponds to $N\nearrow 2$. While Solar System observations rule out values of $\omega$ below a number of order $10^3$, negative values of $\omega$ do arise in approximations to fundamental theories of physics and therefore may play a role in extremely large scale cosmology or at very early times in the evolution of the Universe.

\medskip
\noindent\emph{Acknowledgements.} The work of EW was supported by a Discovery Grant RGPIN 203614 from the Natural Sciences and Engineering Research Council (NSERC). The work of WW was supported by a grant from the Simons Foundation (\#355608, William Wylie). This work arose out of a discussion at Oberwolfach mini-workshop 1440b: Einstein metrics, Ricci solitons and Ricci flow under symmetry assumptions. Both authors wish to thank the organizers.

\section{Focusing and singularities}
\setcounter{equation}{0}

\subsection{The Raychaudhuri equation}
\noindent The focusing behavior of timelike geodesic congruences issuing orthogonally from a spacelike hypersurface $\Sigma$ in a spacetime is studied by means of a scalar Riccati equation, often called the Raychaudhuri equation in general relativity. Let $\gamma$ belong to such a congruence ${\mathcal C}$. We parametrize geodesics in ${\mathcal C}$ by their proper time $t$, so elements $\gamma$ of ${\mathcal C}$ are unit speed timelike geodesics, meaning that $g(\gamma',\gamma')=-1$ where $\gamma'=\frac{d\gamma}{dt}$. At $\Sigma$ we have $\gamma'\vert_{\Sigma}=\nu$ where $\nu$ is the future directed unit normal vector field for $\Sigma$. The congruence ${\mathcal C}$ is surface-forming, so for a curve $\gamma\in {\mathcal C}$, we obtain a foliated neighborhood ${\mathcal N}$ in spacetime near $\gamma:[0,T)\to M$ by moving a parameter distance $t<T$ along the congruence from $\Sigma$, provided that $\gamma$ has no focal point to $\Sigma$ in ${\mathcal N}$. The leaves are also spacelike hypersurfaces. The \emph{extrinsic curvature} or \emph{second fundamental form} of the hypersurface $\Sigma_t$ can be defined as
\begin{equation}
\begin{split}
\label{eq2.1} K(t)(X,Y) = -\nu_t\cdot \left ( \nabla_X Y\right )\ ,\ X,Y\in T_{\gamma(t)}\Sigma_t \ ,
\end{split}
\end{equation}
where $\nu_t=\gamma'(t)$ is the future directed unit normal for $\Sigma_t$. The \emph{expansion scalar} or \emph{mean curvature} of the congruence is
\begin{equation}
\label{eq2.2} H(t):=\tr_h K(t) \ ,
\end{equation}
where $h:=g+\nu\otimes\nu$ is the induced metric on the leaf. Then the Raychaudhuri equation is
\begin{equation}
\label{eq2.3} \frac{\partial H}{\partial t}=-\ric(\nu,\nu)-\vert K \vert^2 =-\ric(\nu,\nu)-|\sigma|^2-\frac{H^2}{(n-1)}\ ,
\end{equation}
where $|K|^2:=h^{ij}h^{kl}K_{ik}K_{jl}$, $\sigma_{ij}:=K_{ij}-\frac{H}{(n-1)}h_{ij}$ is the \emph{shear} (i.e., the tracefree part of $K_{ij}$), and $n$ is the spacetime dimension.

The \emph{Bakry-\'Emery modified mean curvature}, or \emph{$f$-mean curvature}, is defined along our unit speed timelike geodesic congruence to which $\gamma$ belongs by
\begin{equation}%\begin{split}
\label{eq2.4}
H_f:=H - \nabla_{\gamma'} f \equiv H -  f' \ ,
%\end{split}
\end{equation}
where we abbreviate $f\circ\gamma$ by simply writing $f$, so that $\frac{df}{dt}:=f'(t):=(f\circ\gamma)'(t)$. We sometimes write $f_p(t)$ to denote $f\circ\gamma(t)$ where $\gamma$ is the geodesic in ${\mathcal C}$ with initial point $p=\gamma(0)$. The Raychaudhuri equation (\ref{eq2.3}) becomes
\begin{equation}
\label{eq2.5}
\begin{split}
{H_f}'=&\, -\ric^N_f(\gamma',\gamma')-|\sigma|^2-\frac{H^2}{n-1}-\frac{f'^2}{(N-n)}\\
=&\, -\ric^N_f(\gamma',\gamma')-|\sigma|^2-\frac{1}{(n-1)}\left [ H_f^2+2H_f f'+\frac{(1-N)}{(n-N)}f'^2\right ]\ .
\end{split}
\end{equation}
It is convenient to introduce the \emph{normalized $f$-mean curvature}
\begin{equation}
\label{eq2.6} x:=H_f/(n-1)\ .
\end{equation}
Equation (\ref{eq2.5}) then becomes
\begin{equation}
\label{eq2.7} x'= -\frac{1}{(n-1)}\left ( \ric^N_f(\gamma',\gamma')+|\sigma|^2\right ) -x^2 -\frac{2xf'}{(n-1)}-\frac{(1-N)}{(n-1)^2(n-N)}f'^2\ .
\end{equation}

The qualitative features of solutions of the Raychaudhuri equation (\ref{eq2.7}) for $N<1$ are similar to those for $N>n$, owing to the sign of the coefficient of the $f'^2$ term. Hence the $N<1$ and $N>n$ cases are quite similar; nonetheless, in the former case we will need an assumption to control $f$ that is not needed in the latter case. The borderline $N=1$ case has no $n=N$ analogue in this comparison since $N=n$ does not make sense in equation (\ref{eq2.7}). The $N=1$ case has distinct behavior with regard to splitting phenomena.

In what follows, we will introduce the notation $x_p(t):=x\circ\gamma(t)$ to denote the normalized $f$-mean curvature of the leaf $\Sigma_t$ at a point reached by traversing a unit speed timelike geodesic $\gamma$ for a proper time $t$ starting from $\gamma(0)=p\in\Sigma$ with $\gamma'(0)$ orthogonal to $\Sigma$. Then for each $p\in\Sigma$, (\ref{eq2.7}) is an ordinary differential equation for $x_p$. Also, it will sometimes be convenient to reparametrize $\gamma$ using the new parameter
\begin{equation}
\label{eq2.8}
s_p(t):=\int\limits_0^t e^{-\frac{2f_p(\tau)}{(n-1)}}d\tau
\end{equation}
which arises in Definition \ref{definition1.2} and which is obviously monotonic along $\gamma$.

\subsection{A useful lemma}

\noindent While certain singularity theorems imply only that there is an incomplete geodesic (as occurs in a black hole spacetime), our cosmological-type singularity theorems state that \emph{every} future-timelike geodesic is incomplete. These theorems will depend on the following known lemma, stated here for convenient reference. A proof can be found in \cite{GW}.

\begin{lemma}[{\cite[Lemma 2.4]{GW}}]\label{lemma2.1}
Suppose that $S$ is a spacelike Cauchy surface and $\sigma$ is a future complete timelike geodesic. Then there is an arbitrarily long future timelike geodesic $\gamma$ leaving $S$ orthogonally and having no focal point to $S$.
\end{lemma}

\subsection{Non-negative $N$-Bakry-\'Emery-Ricci curvature}

\begin{lemma}\label{lemma2.2}
Let $\gamma$ be a future-complete timelike geodesic with $\gamma(0)=p$. Suppose that
\begin{enumerate}
\item[(i)] $(M,g)$ obeys $\tcd(0,N)$ for some fixed $N\le 1$ or $N=\infty$,
\item[(ii)] $s_p(t)\to\infty$ as $t\to\infty$ (so $\gamma$ is future $f$-complete), and
\item[(iii)] there is a $\delta_p>0$ such that $x_p(0)\le -\delta_p$.
\end{enumerate}
Then there exists a $t_p>0$ such that $x_p(t)\to -\infty$ at or before $t_p$, and for a given function $f_p$ and dimension $n$, $t_p$ depends only on $\delta_p$. Indeed, we may take $t_p$ to be the unique value such that
\begin{equation}
\label{eq2.9}
s_p(t_p)=\frac{1}{\delta_p}e^{-\frac{2f_p(0)}{(n-1)}}\ .
\end{equation}
\end{lemma}

\begin{proof}
Let $\sigma:[0,T)\to M$ be a future-timelike inextendible geodesic with $\sigma(0)=p$, where $T\in (0,\infty]$, and $T_0\le T$ is the first time for which $x_p(t)=0$; if there is no such time then set $T_0=T$. Using $N\le 1$ or $N=\infty$ and applying $\tcd(0,N)$ to equation (\ref{eq2.7}), we obtain the inequality
\begin{equation}
\label{eq2.10} x_p'\le -x_p^2 -\frac{2xf_p'}{(n-1)}
\end{equation}
along $\sigma$.  Since $t<T_0$, $x(t)$ is negative, inequality (\ref{eq2.10}) is equivalent to
\begin{equation}
\label{eq2.11}
\left ( \frac{e^{-\frac{2f_p(t)}{(n-1)}}}{x_p(t)}\right )'\ge e^{-\frac{2f_p(t)}{(n-1)}}\ .
\end{equation}
Integrating this along $\sigma$ from $0$ to some $t<T_0$, we obtain
\begin{equation}
\label{eq2.12}
\frac{e^{-\frac{2f_p(t)}{(n-1)}}}{x_p(t)}-\frac{e^{-\frac{2f_p(0)}{(n-1)}}}{x_p(0)}\ge \int\limits_0^{t_p} e^{-\frac{2f_p(\tau)}{(n-1)}}d\tau = s_p(t)\ ,
\end{equation}
or
\begin{equation}
\label{eq2.13}
x_p(t)\le -\frac{e^{-\frac{2f_p(t)}{(n-1)}}}{\frac{1}{\delta}e^{-\frac{2f_p(0)}{(n-1)}}-s_p(t)}\ .
\end{equation}
From this, we see that $T_0=T$.

By condition (ii) and elementary considerations, equation (\ref{eq2.9}) will have a solution $t_p$ along $\sigma$ if the domain of $\sigma$ extends far enough, a condition which is met for $\sigma=\gamma$; i.e., if the domain of $\sigma$ is $[0,\infty)$. Then we can take $t\nearrow t_p$, causing the denominator in (\ref{eq2.13}) to diverge to $+\infty$ and proving the claim. \end{proof}

\begin{corollary}\label{corollary2.3}
Lemma \ref{lemma2.2} holds with assumption (ii) replaced by
\begin{enumerate}
\item[(ii$'$)] $f_p\le k$ for some $k\in (0,\infty)$.
\end{enumerate}
\end{corollary}

\begin{proof}
By (\ref{eq2.8}), condition (ii$'$) implies condition (ii) of the original lemma.
\end{proof}

\begin{lemma}\label{lemma2.4}
Lemma \ref{lemma2.2} holds also for $N\in (n,\infty)$, and then assumption (ii) is not required. Then $t_p\le (N-1)/\delta$.
\end{lemma}

\begin{proof}
This is the content of \cite[Proposition 3.2]{Case}, with $m=N-n$. The proof proceeds from the identity
\begin{equation}
\label{eq2.14}
\frac{H^2}{(n-1)}+\frac{f'^2}{(N-n)}\ge \frac{(H-f')^2}{(N-1)}=\frac{H_f^2}{(N-1)} \ ,
\end{equation}
which is valid for $N>n$. Using it in the first line of (\ref{eq2.5}), we can replace (\ref{eq2.10}) by $H_f'\le -H_f^2/(N-1)$. As before, for as long as $H_f$ does not cross zero, we can integrate this to obtain $H_f(t)\le \frac{(N-1)}{t-(N-1)/\delta}$, which shows that $H_f$ does not cross zero but instead diverges to $-\infty$ as $t\nearrow T$ for some $T \le (N-1)/\delta$ as long as the timelike geodesic $\gamma$ extends this far, and by assumption it does.
\end{proof}

With these results in hand, the proof of Theorems \ref{theorem1.4} follows along precisely the same lines as the proofs of the corresponding theorems in \cite{GW}.

\begin{proof}[Proof of Theorem \ref{theorem1.4}]
By assumption, conditions (i) and (iii) of Lemma \ref{lemma2.2} hold. Indeed, by compactness, assumption (iii) holds for each $p\in S$ with $\delta_p$ replaced by some $\delta<0$ independent of $p$. When $N\in (-\infty,1]\cup \{ \infty \}$, condition (ii) holds along future-complete timelike geodesics orthogonal to $S$. Then by Lemma \ref{lemma2.2}, or Lemma \ref{lemma2.4} if $N>n$, every future-complete timelike geodesic issuing orthogonally from $S$ focuses within some finite time which depends only on $\delta$. But by Lemma \ref{lemma2.1}, if $(M,g)$ were to admit a future complete timelike geodesic, then there would be a nonfocusing future timelike geodesic of arbitrary length issuing orthogonally from $S$. This is a contradiction, so $(M,g)$ cannot admit a future complete timelike geodesic.
\end{proof}

\subsection{The de Sitter-like case}

\noindent We now consider instead a negative lower bound for the $N$-Bakry-\'Emery-Ricci tensor. To obtain singularity theorems in this case, we will need a concavity assumption on the initial surface.

\begin{lemma}\label{lemma2.5}
As above, let $\gamma$ be a future-timelike geodesic with $\gamma(0)=p$. Suppose that
\begin{enumerate}
\item[(i)] $(M,g)$ obeys $\tcd\left (-(n-1) e^{\frac{-4f}{n-1}},N\right )$ for some fixed $N\le 1$ or $N= \infty$,
\item[(ii)] along $\gamma$, $s_p(t)\to \infty$ at some finite value of $t$, and
\item[(iii)] $x_p(0) \le -(1+\delta_p )e^{-\frac{2f_p(0)}{(n-1)}}$ for some $\delta_p>0$.
\end{enumerate}
Then there exists a $t_p>0$ such that $x_p(t)\to -\infty$ at or before $t_p$, and which depends only on $\delta_p$ (if $N$, $n$ are fixed).
\end{lemma}

\begin{proof}
Using $N\le 1$ or $N= \infty$ and applying $\tcd\left (-(n-1)e^{\frac{-4f}{n-1}},N\right )$ to equation (\ref{eq2.7}), this time we obtain the inequality
\begin{equation}
\label{eq2.15}
\begin{split}
&\, x'\le e^{\frac{-4f}{n-1}}-x^2 -\frac{2xf'}{(n-1)}\\
\Rightarrow\quad &\, \left ( e^{\frac{2f}{n-1}} x\right )'\le e^{\frac{-2f}{n-1}}-x^2e^{\frac{2f}{n-1}}=e^{-\frac{2f}{n-1}} \left ( 1-e^{\frac{4f}{n-1}}x^2\right )\ .
\end{split}
\end{equation}
Writing $y:=e^{\frac{2f}{(n-1)}} x$, this becomes
\begin{equation}
\label{eq2.16}
\begin{split}
y'\le &\, e^{\frac{-2f}{(n-1)}}\left ( 1-y^2\right )\\
\Rightarrow\quad {\dot y} \le &\, \left ( 1-y^2\right )\ ,
\end{split}
\end{equation}
where the dot over the $y$ indicates differentiation with respect to $s=s_p(t)$. Note that $y(0)<-1$. Integrating over an interval small enough so that $y(t)<-1$, we obtain
\begin{equation}
\label{eq2.17}
\begin{split}
y \le &\, -\coth \left (t_p -s\right )\ , \\
\Rightarrow\quad x \le &\, -e^{-\frac{2f_p(t)}{(n-1)}}\coth \left (t_p -s\right )\ , \\
t_p :=&\, \arctanh\left(\frac{1}{1+ \delta_p}\right)\ .
\end{split}
\end{equation}
Thus, $y<-1$ throughout its domain of definition and $y\to -\infty$ (thus $x_p\to -\infty$) on approach to some $t\le t_p$, where $t_p$ depends only on $\delta_p$ (for fixed $N$ and $n$).
\end{proof}

\begin{corollary}\label{corollary2.6}
Suppose that
\begin{enumerate}
\item[(i$'$)] $(M,g)$ obeys $\tcd\left (-(n-1),N\right )$ for some fixed $N\le 1$ or $N=\infty$,
\item[(ii$'$)] $f_p\le k$ along $\gamma$, for some $k\in (0,\infty)$, and
\item[(iii$'$)] $x_p(0) \le -(1+\delta_p )e^{\frac{2(k-f_p(0))}{(n-1)}}$ for some $\delta_p>0$.
\end{enumerate}
Then there exists a $t_p=t_p(\delta_p)>0$ such that $x_p(t)\to -\infty$ at or before $t_p$. Furthermore, we can replace conditions (ii$'$) and (iii$'$) by
\begin{enumerate}
\item[(ii$''$)] $\nabla f$ is future-causal, and
\item[(iii$''$)] $x_p(0) \le -(1+\delta_p )$.
\end{enumerate}
\end{corollary}

\begin{proof}
Define ${\bar f}:=f-k$. By (ii$'$), we have ${\bar f}\le 0$, so $e^{-\frac{4{\bar f}}{(n-1)}}\ge 1$. Combining this with (i$'$), we see that that $\tcd\left (-(n-1) e^{\frac{-4{\bar f}}{(n-1)}},N\right )$ holds. We also have that ${\bar s}(t):=\int\limits_0^t e^{-\frac{2{\bar f}(\tau)}{(n-1)}}d\tau \ge \int\limits_0^t d\tau=t$, which diverges as $t\to \infty$. Finally, (iii$'$) implies that $x_p(0) \le -(1+\delta_p )e^{-\frac{2{\bar f}_p(0)}{(n-1)}}$. Now apply Lemma \ref{lemma2.5} to $(M,g,{\bar f})$. This proves the first part.

Next, if $\nabla f$ is future-causal, then $f$ is decreasing along any future-timelike curve, so $f_p(t)\le f_p(0)=:k$, and then $-(1+\delta_p )e^{\frac{2(k-f_p(0))}{(n-1)}}=-(1+\delta_p )$, showing that conditions (ii$''$), (iii$''$) imply conditions (ii$'$), (iii$'$).
\end{proof}

Finally, just as with Lemma \ref{lemma2.2} and Lemma \ref{lemma2.4}, there is an $N>n$ version of Lemma \ref{lemma2.5} that holds without any assumption controlling $f$.

\begin{lemma}\label{lemma2.7}
For some fixed $N>n$, suppose that
\begin{enumerate}
\item[(i)] $(M,g)$ obeys $\tcd\left (-(n-1),N\right )$ and
\item[(ii)] at $p$ we have $x_p(0) \le -(1+\delta_p )$ for some $\delta_p>0$.
\end{enumerate}
Then there exists a $t_p=t_p(\delta_p)>0$ such that $H_f(t)\to -\infty$ along $\gamma$ at or before $\gamma(t_p)$.
\end{lemma}

\begin{proof}
Combining the first line of (\ref{eq2.5}), the identity (\ref{eq2.14}), and assumption (i), we have
\begin{equation}
\label{eq2.18}
H_f'\le n-1-\frac{H_f^2}{N-1} < N-1-\frac{H_f^2}{N-1}\ .
\end{equation}
We may integrate as before and use that $H_f(0)=(n-1)x_p(0)\le -(n-1)(1+\delta_p )$ to obtain
\begin{equation}
\label{eq2.19}
\begin{split}
H_f< &\, -(N-1)\coth\left ( t_p-t\right )\ ,\\
t_p= &\, \arctanh \frac{(N-1)}{(n-1)(1+\delta_p)} \ ,
\end{split}
\end{equation}
from which the claim follows.
\end{proof}

With these results, we are now in a position to prove Theorems \ref{theorem1.6} and \ref{theorem1.7}.

\begin{proof}[Proof of Theorem \ref{theorem1.6}] When $N\le 1$ or $N=\infty$ then by assumptions (b.i), (b.ii), and (b.iii) and the compactness of $S$, assumptions (i--iii) of Lemma \ref{lemma2.5} hold, with assumption (ii) applying to future-complete timelike geodesics $\gamma$ orthogonal to $S$. If instead we have $N>n$, then assumptions (a.i) and (a.ii) and compactness of $S$ imply that the assumptions of Lemma \ref{lemma2.7} are verified. In either case, every future-complete timelike geodesic issuing orthogonally from $S$ then must have a focal point to $S$ within some finite time which depends only on $\delta$. But then the existence of a future-complete timelike geodesic would lead to a contradiction with Lemma \ref{lemma2.1}, as in the proof of Theorem \ref{theorem1.4}.
\end{proof}

\begin{proof}[Proof of Theorem \ref{theorem1.7}]
The assumptions of this theorem imply that the assumptions of Corollary \ref{corollary2.6} hold, which in turn imply as before that every future-complete timelike geodesic issuing orthogonally from $S$ then must have a focal point to $S$ within some finite time which depends only on $\delta$. Once again, the existence of a future-complete timelike geodesic would lead to a contradiction with Lemma \ref{lemma2.1}.
\end{proof}

\section{Rigidity}
\setcounter{equation}{0}

\noindent We now consider the case of equality in the mean curvature assumption (\ref{eq1.3}) and in the analogous assumption in Theorem \ref{theorem1.7}. In \cite{GW}, the main idea was to employ an \emph{extrinsic curvature flow} to deform the mean curvature slightly in an effort to restore a strict inequality so that the singularity theorems continue to apply. This fails only if the geometry is quite special, generally a product or warped product, which produces the desired rigidity statement.

The extrinsic curvature flow to be employed is defined by choosing a function $\varphi$ and writing
\begin{equation}
\label{eq3.1}
\begin{split}
\frac{\partial F}{\partial r}=&\,\, \varphi\nu\ ,\\
F(0,\cdot)=&\,\, \id\ .
\end{split}
\end{equation}
Here $F(r,\cdot):\Sigma\hookrightarrow M$ is a family of embeddings, $\nu$ is the corresponding timelike unit normal field, and $r$ is the family parameter. The function $\varphi$ depends on the mean curvature $H(r,\cdot)$ of $F(r,\cdot )$. The choice made in \cite{GW} is
\begin{equation}
\label{eq3.2} \varphi=H_f-\lambda=H-\nabla_{\nu}f-\lambda\ ,
\end{equation}
where $\lambda$ is a constant. Such a solution is called a \emph{$(\lambda,f)$-mean curvature flow}, and reduces to the familiar \emph{mean curvature flow} when $f=\lambda=0$.

The technique employed in \cite{GW} was to construct a suitable deformation $\varphi$ by analyzing the evolution equation
\begin{eqnarray}
\label{eq3.3}
\frac{\partial \varphi}{\partial r}&=&\Delta_{\Sigma_r}\varphi-D_{\Sigma_r} f \cdot D_{\Sigma_r} \varphi +c\varphi\ ,\\
\label{eq3.4}
c&=&-|K|_{h_r}^2-\ric_f(\nu,\nu)
\end{eqnarray}
where $\Delta_{\Sigma_r}\varphi := D_{\Sigma_r}\cdot D_{\Sigma_r}\varphi$ is the Laplacian (the trace of the Hessian formed from the Levi-Civita connection $D_{\Sigma_r}$ of the induced metric $h_{ij}(r)$) of $\varphi$ on $\Sigma_r:=(\Sigma,h_{ij}(r))$ and $D_{\Sigma_r} f \cdot D_{\Sigma_r} \varphi =h(r)(D_{\Sigma_r} f, D_{\Sigma_r} \varphi)$, but $\ric_f$ is the Bakry-\'Emery tensor of the ambient spacetime. We want to replace this with the $N$-Bakry-\'Emery tensor. As well, we will expand the second fundamental form $K$ in terms of its tracefree part $\sigma$ and its trace $H$, and replace the latter by $H_f$. We get
\begin{equation}
\label{eq3.5}
\begin{split}
c=&\, -|\sigma|^2-\frac{H^2}{(n-1)}-\ric_f(\nu,\nu)\\
=&\, -|\sigma|^2-\frac{1}{(n-1)}\left ( H_f+f'\right )^2-\ric_f^N(\nu,\nu)+\frac{f'^2}{(n-N)}\\
=&\, -|\sigma|^2-\frac{1}{(n-1)} \left ( H_f^2+2f'H_f\right ) -\ric_f^N(\nu,\nu)-\frac{(1-N)f'^2}{(n-1)(n-N)}\ .
\end{split}
\end{equation}

\begin{lemma}\label{lemma3.1}
Let $(\Sigma,h_{ij}^0)\hookrightarrow (M,g)$ be a closed spacelike hypersurface such that $\varphi:=H_f-\lambda\le 0$ for all $p\in\Sigma$. There is an $\varepsilon>0$ such that the $(\lambda,f)$-mean curvature flow $F:[0,\varepsilon)\times\Sigma\to (M,g)$ obeying (\ref{eq3.1}, \ref{eq3.2}) exists. Furthermore, either $\varphi(r,q)<0$ for all $r\in(0,\varepsilon)$ and all $q\in\Sigma$ or $\varphi\equiv 0$ for all $r\in[0,\varepsilon)$ and all $q\in\Sigma$. In particular, if $\varphi(0,p)<0$ for some $p\in\Sigma$, then $\varphi(r,q)<0$ for all $r\in(0,\varepsilon)$ and all $q\in\Sigma$.
\end{lemma}

\begin{proof}
For $\Sigma$ a closed spacelike hypersurface, \cite[Theorem 2.5.19]{Gerhardt}
guarantees a smooth solution of (\ref{eq3.1}, \ref{eq3.2}) on $[0,\varepsilon)\times\Sigma$ for some $\varepsilon>0$.

Define $u:=e^{-ar}\varphi$, where $a\ge \max_{[0,\varepsilon)\times M}c$ (choosing a smaller $\varepsilon$ if necessary). Then (\ref{eq3.3}) becomes
\begin{equation}
\label{eq3.6}
\frac{\partial u}{\partial r}=\Delta_{\Sigma_r}u-D_{\Sigma_r} f \cdot D_{\Sigma_r} u +(c-a)u\ .\\
\end{equation}
We have $\varphi\le 0$ at $r=0$, so $u\le 0$ at $r=0$. Since $c-a\le 0$, the strong maximum principle \cite[Theorem 2.7]{Lieberman} implies that $u\le 0$, so $\varphi\le 0$ for all $r\in [0,\epsilon)$ and either $\varphi<0$ for all $r\in (0,\varepsilon)$ or $\varphi\equiv 0$.
\end{proof}

Given Lemma \ref{lemma3.1}, the proofs of Theorems \ref{theorem1.9} and \ref{theorem1.10} follow just as the analogous results follow in \cite{GW}, with the exception of the $N=1$ case in Theorem \ref{theorem1.9}.

\begin{proof}[Proof of Theorem \ref{theorem1.9}.]
We introduce Gaussian normal coordinates in a neighborhood $U$ of $S$ in $J^+(S)$,
\begin{equation}
\label{eq3.7} g = -dt^2 + h_{ij}dx^idx^j\ ,\ t \in [0,\epsilon)\ ,
\end{equation}
and let $x(t)=H_f(t)/(n-1)$ as above. Then, using $N\in (-\infty,1]\cup(n,\infty)\cup \{ \infty \}$ and $\tcd(0,N)$ in (\ref{eq2.7}), $x$ obeys
\begin{equation}
\label{eq3.8} x'+\frac{2f'}{(n-1)}x\le -x^2\ ,\ x(0)\le 0\ .
\end{equation}
Multiplying by $e^{2f/(n-1)}$ and integrating to the future along the $t$-geodesics yields
\begin{equation}
\label{eq3.9} e^{\frac{2f(t)}{(n-1)}}x(t)-e^{\frac{2f(0)}{(n-1)}}x(0)=-\int\limits_0^t e^{\frac{2f(u)}{(n-1)}}x^2(u)du\le 0\ .
\end{equation}
Using $x(0)\le 0$, we obtain that $x(t)\le 0$ and thus $H_f(t)\le 0$ for all $t\ge 0$ in $U$.

If $H_f(t_0)<0$ everywhere on a $t=t_0$ Cauchy surface, then by Theorem \ref{theorem1.4} every timelike geodesic will be future incomplete, contrary to assumption. If, however, there are both points where $H_f(t_0)=0$ and points where $H_f(t_0)<0$, then the $t=t_0$ hypersurface can serve as initial data for a $(\lambda,f)$-mean curvature flow (\ref{eq3.1}, \ref{eq3.2}) with $\lambda=0$ on an interval $s\in [0,\varepsilon)$, yielding deformed hypersurfaces with $H_f<0$ everywhere according to Lemma \ref{lemma3.1}. Furthermore, the deformed hypersurfaces are spacelike Cauchy surfaces. Then we can apply Theorem \ref{theorem1.4} using a deformed Cauchy surface as the initial surface, again yielding incomplete geodesics.

The remaining possibility is that there is no $t=t_0$ Cauchy surface in $U$ on which $H_f$ differs from $0$; i.e., $H_f(t) = 0$ for all $t \in [0,\epsilon)$. Then by (\ref{eq2.7}) we have that $\ric_f^N(\gamma',\gamma')=0$ and $\sigma=0$ throughout the domain, and either $f'=0$ as well or $N=1$. In the former case, since $H_f=0$ and $f'=0$ we then obtain $H=0$ and so the domain admits a foliation by totally geodesic Cauchy surfaces, yielding the splitting as claimed. Because the geodesics $\gamma$ orthogonal to the Cauchy surface extend indefinitely, the splitting is global to the future, and also clearly $f$ is constant.

However, if $N=1$, then we cannot conclude that $H$ or $f'$ vanish. From $H_f=0$, we have only that $H=f'$ for every $t=const$ hypersurface in the coordinate domain, and since $\sigma=0$ then the metric (\ref{eq3.7}) on that domain splits as a \emph{twisted product}
\begin{equation}
\label{eq3.10}
ds^2 =-dt^2+e^{2f/(n-1)}{\hat h}
\end{equation}
for some metric ${\hat h}$ on $S$. Since $f$ is a function on the Cauchy surfaces as well as a function of $t$, this is not yet a warped product. However, we may appeal to \cite[Proposition 2.2]{Wylie}, which argues as follows. For this metric and for $\frac{\partial}{\partial y^{\alpha}}\in TS$, the Gauss-Codazzi-Mainardi equations yield
\begin{equation}
\label{eq3.11}
\ric\left (\frac{\partial}{\partial t},\frac{\partial}{\partial y^{\alpha}} \right ) =-\frac{(n-2)}{(n-1)} \frac{\partial H}{\partial y^{\alpha}} =-\frac{(n-2)}{(n-1)} \frac{\partial^2 f}{\partial t \partial y^{\alpha}},
\end{equation}
while a simple calculation yields
\begin{equation}
\label{eq3.12}
\hess f  \left (\frac{\partial}{\partial t},\frac{\partial}{\partial y^{\alpha}}\right )+ \frac{1}{(n-1)}\left \langle \frac{\partial}{\partial t}, df\right \rangle \left \langle \frac{\partial}{\partial y^{\alpha}}, df\right \rangle=\frac{\partial^2 f}{\partial t \partial y^{\alpha}}\ .
\end{equation}
Then
\begin{equation}
\label{eq3.13}
\ric_f^1 \left (\frac{\partial}{\partial t},\frac{\partial}{\partial y^{\alpha}} \right ) = \frac{1}{(n-1)}\frac{\partial^2 f}{\partial t \partial y^{\alpha}}\ .
\end{equation}
But since $\ric_f^1\ge 0$ and $\ric_f^1\left ( \gamma',\gamma'\right )=\ric_f^1\left ( \frac{\partial}{\partial t},\frac{\partial}{\partial t}\right )=0$, then $\ric_f^1 \left (\frac{\partial}{\partial t},v \right )=0$ for any $v\in \left (\frac{\partial}{\partial t}\right )^{\perp}$. Thus $\frac{\partial^2 f}{\partial t \partial y^{\alpha}}=0$, so $f(t,y)=\psi(t)+\phi(y)$ and (\ref{eq3.10}) assumes the desired warped product form with $h=e^{2\phi/(n-1)}{\hat h}$.
\end{proof}

\begin{proof}[Proof of Theorem \ref{theorem1.10}]
Say $N\le 1$ or $N=\infty$. Using $\ric_f^N\ge -(n-1)$ in (\ref{eq2.7}), we get that the normalized $f$-mean curvature $x(t):=H_f(t)/(n-1)$ satisfies $x'\le 1-x^2-\frac{2xf'}{(n-1)}$, with $x(0) \le -1$. Furthermore, since $\nabla f$ is future-causal, for $\epsilon$ sufficiently small so that $x < 0$ for all $t\in [0,\epsilon)$ we have that $xf' \ge 0$, and so $x'\le 1-x^2$ for small enough $t$. Then elementary comparison with the solution to $y' = 1 -y^2$, $y(0) = -1$, implies that $x(t) \le -1$ for all $t\in [0,\epsilon)$, so $H_f(t) \le -(n-1)$ for all $t \in [0,\epsilon)$. If $N>n$, we may draw the same conclusion from (\ref{eq2.19}) (without requiring $\nabla f$ to be future-causal).

If, for some $t_0$, $H_f(t_0)$ is strictly less than $-(n-1)$ at some point but not at every point in the $t_0$ hypersurface, then we can employ a $(\lambda,f)$-mean curvature flow (\ref{eq3.1}, \ref{eq3.2}), this time with $\lambda=-(n-1)$, and invoke Lemma \ref{lemma3.1} to obtain a nearby spacelike Cauchy surface with $f$-mean curvature $H_f<-(n-1)$ pointwise.

Having obtained an $H_f<-(n-1)$ Cauchy surface, we can employ Theorem \ref{theorem1.7} (if $N\in (-\infty,1] \cup \{ \infty \}$) or Theorem \ref{theorem1.6}.(a) (if $N>n$), using this Cauchy surface as the initial hypersurface for the geodesic congruence. This implies that every timelike geodesic will be future incomplete, contrary to assumption.

Thus, $H_f(t) = -(n-1)$ for all $t \in [0,\epsilon)$, and so $x=-1$ in (\ref{eq2.7}). Writing the $\tcd\left ( -(n-1),N\right )$ condition as $\ric_f^N(\gamma',\gamma') = -(n-1) +\delta^2$ for some function $\delta(t,y)$, then (\ref{eq2.7}) yields
\begin{equation}
\label{eq3.14}
0= -\delta^2-|\sigma|^2+2f'-\frac{(1-N)}{(n-1)(n-N)}f'^2\ .
\end{equation}
Since $N\in (-\infty,1]\cup (n,\infty)\cup \{ \infty \}$ and since $f'\le 0$ for all such $N$, each individual term on the right must vanish, so $\sigma=0$ and $\delta=f'=0$. Combining $f'=0$ with $x=-1$, we obtain $H=-(n-1)$ for $t\in [0,\epsilon)$. Then we conclude that $ds^2=-dt^2+e^{-2t}h$ and since $f$ is time-independent and $\nabla f$ is future-causal, $f$ is constant. But as before, since the geodesics tangent to $\frac{\partial}{\partial t}$ are future-complete, the splitting is in fact global to the future: we may take $\epsilon\to\infty$.
\end{proof}

\begin{proof}[Proof of Theorem \ref{theorem1.11}]
When $N>n$, arguing as in (\ref{eq2.18}, \ref{eq2.19}), $\ric_f^N\ge -(N-1)$ implies that we have $H_f' \leq (N-1) - \frac{(H_f)^2}{N-1}$. In turn, this and the assumption that $H_f(0)\le -(N-1)$ imply $H_f(t) \le -(N-1)$ for all $t \in [0,\epsilon)$.  Then, arguing as above in the proof of Theorem \ref{theorem1.10}, using the assumptions that the geodesics orthogonal to the Cauchy surface are future complete and that $\tcd(-(N-1),N)$ holds and invoking Theorem \ref{theorem1.6}.(a), we obtain that $H_f(t) = -(N-1)$ for all $t \in [0,\epsilon)$.

Combining the first line of (\ref{eq2.5}), inequality (\ref{eq2.14}) and the $\tcd(-(N-1),N)$ assumption, we obtain the inequality
\begin{equation}
\label{eq3.15}
{H_f}'= -\ric^N_f(\gamma',\gamma')-|\sigma|^2-\frac{H^2}{n-1}-\frac{f'^2}{(N-n)} \le  -(N-1)- \frac{H_f^2}{N-1}\ .
\end{equation}
Since $H_f\equiv -(N-1)$, we must have equality in \ref{eq3.15}. In particular, we must have $\ric^N_f(\gamma',\gamma')=-(N-1)$, $\sigma = 0$, and equality in (\ref{eq2.14}).  Equality in (\ref{eq2.14}) implies that $H = - \frac{n-1}{N-n}f'$.  Combining this with  $H_f = -(n-1)$ implies that $f' = N-n$ and $H = -(n-1)$.   Since $\sigma = 0$ this implies that $ds^2=-dt^2+e^{-2t}h$ and that $f = (N-n)t + f_S$ for  $t \in [0,\epsilon)$. As before, since the geodesics tangent to $\frac{\partial}{\partial t}$ are future-complete, the splitting is in fact global to the future: we may take $\epsilon\to\infty$.
\end{proof}

\section{Final remarks}

\noindent There remain a number of open issues regarding the Lorentzian $N$-Bakry-\'Emery theory. With the purpose of stimulating further research, we list some of them here.

First, we note that in \cite{Case} a Lorentzian timelike splitting theorem analogous to the Cheeger-Gromoll splitting theorem is established for $N>n$ and $N=\infty$. It seems to us quite plausible that this theorem would admit an extension to $N\le 1$, likely again with a partial loss of rigidity for $N=1$.

We also note that in \cite{Case} a Lorentzian Bakry-\'Emery version of the Hawking-Penrose singularity theorem \cite{HE} is established for $N>n$ and $N=\infty$. Again, it seems clear that this result will admit an extension. However, the theory of Jacobi and Lagrange fields along null geodesics differs from that along timelike geodesics because the orthogonal complement to the tangent field of the geodesics contains the tangent field itself. Because components along the tangent direction play no role, one quotients out by this direction. The net effect is that coefficients of $1/(n-1)$ in the Raychaudhuri equation become $1/(n-2)$. This modifies equation (\ref{eq2.7}) so that the critical value for the synthetic dimension will be $N=2$ (which, interestingly, corresponds to $\omega=-3/2$, in Brans-Dicke theory, which is the value at which these theories become undefined). Furthermore, now the appropriate splitting theorem will be analogous to the \emph{null splitting theorem} for Lorentzian geometry \cite{Galloway}. Because of these theoretical differences and potentially new features, this case deserves its own separate treatment.

Finally, in the standard non-Bakry-\'Emery cases (i.e., when $f$ is constant), one can replace pointwise conditions on the Ricci tensor by integral conditions on the Ricci curvature along geodesics (e.g., \cite{Tipler}). To our knowledge, this has not yet been done in the $N$-Bakry-\'Emery case for any $N$, including $N=\infty$, or for either Riemannian or Lorentzian signature.

\end{document}